\documentclass[11pt]{amsart}

\usepackage{amssymb, bbm} 
\usepackage{graphics}
\usepackage{color}
\usepackage{natbib}
\usepackage{amscd}
\usepackage{graphicx}
\usepackage{pb-diagram}
\usepackage[all]{xy}
\usepackage{wrapfig}
\usepackage[mathscr]{eucal}

\usepackage[margin=3cm]{geometry}
\setcitestyle{open={[},close={]},comma,numbers}
\newcommand\C{{\mathbb{C}}}
\newcommand\N{{\mathbb{N}}}
\newcommand\R{{\mathbb{R}}}
\newcommand\T{{\mathbb T}}
\newcommand\Z{{\mathbb Z}}

\newcommand{\dist}{\operatorname{dist}}
\newcommand{\supp}{\operatorname{supp}}

\newcommand{\rank}{\operatorname{rank}}
\newcommand{\codim}{\operatorname{codim}}
\newcommand\BL{\operatorname{BL}}

\newcommand\HK{\operatorname{HK}}
\newcommand{\om}{\omega}

\newcommand{\fm}{{\mathfrak m}}

\let\oldepsilon\epsilon
\let\oldvarepsilon\varepsilon
\renewcommand{\epsilon}{\oldvarepsilon}
\renewcommand{\varepsilon}{\oldepsilon}

\theoremstyle{plain}
  \newtheorem{theorem}{Theorem}[section]

  \newtheorem{proposition}[theorem]{Proposition}
  \newtheorem{prop}[theorem]{Proposition}
  
  \newtheorem{lemma}[theorem]{Lemma}
  \newtheorem{corollary}[theorem]{Corollary}

\theoremstyle{definition}
  \newtheorem*{remark*}{Remark}

  \newtheorem*{question*}{Question}

\numberwithin{equation}{section}
\date{\today}
\subjclass[2010]{42B37, 44A12, 52A40}
\keywords{Discrete inequalities, Fourier transforms, duality}
\title{Fourier duality in the Brascamp--Lieb inequality}

\author[J. Bennett]{Jonathan Bennett}
\address{School of Mathematics, The Watson Building, University of Birmingham, Edgbaston,
Birmingham, B15 2TT, England.}
\email{J.Bennett@bham.ac.uk}

\author[E. Jeong]{Eunhee Jeong}
\address{School of Mathematics, Korea Institute for Advanced Study, Seoul 02455, Republic of Korea} 
\email{eunhee@kias.re.kr}

\begin{document}
\maketitle
\begin{abstract}
It was observed recently in work of Bez, Buschenhenke, Cowling, Flock and the first author, that the euclidean Brascamp--Lieb inequality satisfies a natural and useful Fourier duality property.
The purpose of this paper is to establish an appropriate discrete analogue of this. Our main result identifies the Brascamp--Lieb constants on (finitely-generated) discrete abelian groups with Brascamp--Lieb constants on their (Pontryagin) duals.
As will become apparent, the natural setting for this duality principle is that of locally compact abelian groups, and this raises basic questions about Brascamp--Lieb constants formulated in this generality. 
\end{abstract}
\section{Background and results}
The euclidean Brascamp--Lieb inequality is a broad generalisation of a range of important multilinear functional inequalities in mathematics, and takes the form
\begin{equation}\label{BL}
\left|\int_H f_1\otimes\cdots\otimes f_m\right|\leq\BL(H,\mathbf{p})\prod_{j=1}^m\|f_j\|_{L^{p_j}(H_j)};
\end{equation}
here $H_1,\hdots, H_m$ are euclidean spaces equipped with Lebesgue measure, $H$ is a subspace of the cartesian product $H_{1}\times\cdots\times H_m$, also with Lebesgue measure, and $\mathbf{p}=(p_j)\in [1,\infty]^m$. We refer to $(H,\mathbf{p})$ as the \emph{Brascamp--Lieb data} for the inequality \eqref{BL}. The expression $\BL(H,\mathbf{p})$, referred to as the \emph{Brascamp--Lieb constant}, is used to denote the smallest constant in the inequality \eqref{BL} over all input functions $f_j\in L^{p_j}(H_j)$, and may of course be infinite at this level of generality. This is one of a number of equivalent formulations of the Brascamp--Lieb inequality/constant (see \cite{BBBCF}), chosen here as it naturally relates to the elementary Fourier-invariance property
\begin{equation}\label{Finv}
\int_Hf_1\otimes\cdots\otimes f_m=\int_{H^\perp}\widehat{f}_1\otimes\cdots\otimes\widehat{f}_m;
\end{equation}
here $H^\perp$ denotes the orthogonal complement of $H$ in $H_1\times\cdots\times H_m$, and $\widehat{f}_j$ the $n_j$-dimensional Fourier transform, where $n_j=\dim H_j$. This invariance property is a generalisation of the basic fact that the euclidean Fourier transform maps convolution to pointwise multiplication, and has of course been used by many authors in a variety of contexts -- see, for example, \cite{Ball}, \cite{DT},  \cite{Roth} and \cite{BBBCF}.
Recently a corresponding notion of duality within the set of Brascamp--Lieb data was established in \cite{BBBCF}.
\begin{theorem}[Fourier duality  \cite{BBBCF}]\label{basicthm}
If $H$ is a subspace of $H_1\times\cdots\times H_m$ and $1\leq p_1,\hdots,p_m\leq\infty$, then
\begin{equation}\label{dualityid}
\BL(H,\mathbf{p})=B_{\mathbf{p}}\BL(H^\perp,\mathbf{p}'),
\end{equation}
where 
\begin{equation}\label{becks}
B_{\mathbf{p}}=\prod_{j=1}^m \left(\frac{p_j^{1/p_j}}{(p_j')^{1/p_j'}}\right)^{n_j/2}
\end{equation}
and $\mathbf{p}'=(p_j')$; here $p_j'$ denotes the conjugate exponent to $p_j$.\footnote{There is a minor technical point here that arises if $p_j=\infty$ for some $j$, and the $j$th projection map from $H$ to $H_j$ (or later on $G_j$) is not surjective. In this case the $L^\infty$ norm should be interpreted as the supremum rather than the essential supremum. We refer to the discussion of Fubini's theorem and its dual in Section \ref{Sect:examples} for a simple example.
}
\end{theorem}
Theorem \ref{basicthm} tells us that, up to an explicit factor depending only on $n_1,\hdots,n_m$ and $\mathbf{p}$, the map
\begin{equation}\label{dualitymap}
(H,\mathbf{p})\mapsto (H^\perp,\mathbf{p}')
\end{equation}
is a constant-preserving involution on the set of Brascamp--Lieb data. In what follows we shall refer to $(H,\mathbf{p})$ and $(H^\perp,\mathbf{p}')$ as \emph{dual data}, and \eqref{dualitymap} as the \textit{duality map}.

One immediate consequence of Theorem \ref{basicthm} and the  Fourier invariance property \eqref{Finv} is the ``Fourier--Brascamp--Lieb inequality"
\begin{equation}\label{FBL}
\left|\int_H f_1\otimes\cdots\otimes f_m\right|\leq \BL(H,\mathbf{p})\prod_{j=1}^m A_{p_j}^{-n_j/2}\|\widehat{f}_j\|_{L^{p_j'}(H_j)},
\end{equation}
where $A_p^{n/2}:=(p^{1/p}/(p')^{1/p'})^{n/2}$ denotes the best constant in the classical $n$-dimensional forward and reverse Hausdorff--Young inequalities \cite{Beckner}, \cite{BrascampLieb}. Depending on where the exponents $p_j$ lie relative to $2$, the inequality \eqref{FBL} may be tighter than \eqref{BL} for \emph{non-gaussian} inputs $f_j$. An instance of this observation was implicitly used by Ball in his work on volumes of sections of cubes \cite{Ball}.

Perhaps the clearest benefit of Theorem \ref{basicthm}  (and our forthcoming Theorem \ref{dualitythm}) is that it allows one to switch between a Brascamp--Lieb datum and its dual, which may be rather more transparent; see \cite{BBBCF} and \cite{Kyoto} for an example of this. Perhaps the simplest interesting example is the duality between the data associated with Young's convolution inequality and H\"older's inequality, which goes some way to explain the nature of the constant $B_{\mathbf{p}}$ in \eqref{dualityid}. We elaborate on this and a variety of other examples in Section \ref{Sect:examples}.

Theorem \ref{basicthm} may be proved rather quickly using some quite heavy-duty machinery: one combines the Fourier-invariance property \eqref{Finv} with an application of Lieb's theorem \cite{L} on the exhaustion of $\BL(H,\mathbf{p})$ by centred gaussian functions $f_j$, along with the elementary fact that the Fourier transform maps centred gaussians to centred gaussians. We refer to \cite{BBBCF} for further explanation.

The Fourier-invariance property captured by \eqref{Finv} is naturally generalised to the context of \textit{locally compact abelian} (LCA) groups. Here $H$ is taken to be a closed subgroup of $G_1\times\cdots\times G_m$, with each $G_j$ a LCA group, and
$H^\perp$ the subgroup of the dual group $\widehat{G}_1\times\cdots\times \widehat{G}_m$ given by
$$
H^\perp=\{\chi\in \widehat{G}_1\times\cdots\times \widehat{G}_m: \chi(h)=1\;\;\mbox {for all }\;\;h\in H\}.$$ Defining $\BL_{\mathbf{G}}(H, \mathbf{p})$ to be the best constant in 
\begin{equation}\label{BLabstract}
\left|\int_H f_1\otimes\cdots\otimes f_m \right|\leq\BL_{\mathbf{G}}(H,\mathbf{p})\prod_{j=1}^m\|f_j\|_{L^{p_j}(G_j)},
\end{equation}
it is natural to ask whether
\begin{equation}\label{conj}
\BL_{\mathbf{G}}(H, \mathbf{p})<\infty \iff \BL_{\mathbf{ \widehat{G}}}(H^\perp, \mathbf{p}')<\infty
\end{equation}
in full generality.
Of course the specific constants $\BL_{\mathbf G}(H,\mathbf p)$ and $\BL_{\mathbf{\widehat{G}}}(H^\perp,\mathbf p')$ depend on how the Haar measures implicit in \eqref{BLabstract} are normalised. There is no canonical choice of normalisation for the $G_j$ or $H$, although once chosen, the normalisations for the $\widehat{G}_j$ and $H^\perp$ are naturally prescribed by insisting that Plancherel's theorem on $G_j$ holds with constant $1$, and that the forthcoming abstract Fourier-invariance property \eqref{poisson} holds for $H$ with constant $1$. 
In the setting of discrete groups -- the focus of this work --  it is natural (and conventional) to assign counting measure to $G$ and $H$, and correspondingly require that the measures on the compact groups $\widehat G$ and $H^\perp$ are normalised to have mass $1$. 

In light of Theorem \ref{basicthm}, the basic structure theory of locally compact abelian groups (see \cite[Theorem 24.30]{HR}) suggests an explicit quantitative form of \eqref{conj}. In particular, since each factor $G_j$ is isomorphic to $\mathbb{R}^{n_j}\times G_{j,0}$, where $G_{j,0}$ has a compact open subgroup, it seems natural to expect that
\begin{equation}\label{dualityidgen}
\BL_{\mathbf{G}}(H,\mathbf{p})=B_{\mathbf{p}}\BL_{\mathbf{\widehat{G}}}(H^\perp,\mathbf{p}'),
\end{equation}
with $B_{\mathbf{p}}$ given by \eqref{becks}.

In this paper we consider the validity of \eqref{dualityidgen}
in the particular context of (finitely-generated) \emph{discrete} groups $G_j$. Here of course the euclidean factor in $G_j$ is trivial for each $j$, so that the constant appearing in \eqref{dualityidgen} is 1.
It should be pointed out that multilinear functionals of the form
\begin{equation}\label{BLform}
(f_1,\hdots, f_m)\mapsto \int_H f_1\otimes\cdots\otimes f_m,
\end{equation} 
which arise frequently in the setting of euclidean spaces $G_j=\mathbb{R}^{n_j}$, are also common in discrete settings.
For discrete (usually finite) groups $G_j$ they appear in additive combinatorics in the form of higher order (or Gowers) inner products -- see Tao \cite{TaoHOFA}. In such discrete contexts the Fourier invariance property \eqref{Finv} also plays an important role, and this dates back at least to the celebrated work of Roth \cite{Roth} on the existence of arithmetic progressions of length three in subsets of $\mathbb{Z}$. For torsion-free discrete groups (integer lattices) such multilinear functionals also naturally arise in induction-on-scales arguments within harmonic analysis (for example in \cite{BB}), and in optimisation-theoretic aspects of communication theory (see \cite{CDKSY0}).
We refer to Section \ref{Sect:examples} for some interpretations of \eqref{dualityidgen} in a selection of settings.

Our main theorem is the following:
\begin{theorem}\label{dualitythm} Let $m\ge 1$, $\mathbf p=(p_j)\in [1,\infty]^m$,  and $G_1,\hdots, G_m$ be finitely-generated abelian groups. Then,
\begin{equation}\label{isome}
 \BL_{\mathbf G}(H,\mathbf p) =\BL_{\widehat{\mathbf G}} (H^\perp, \mathbf p')
\end{equation}
for all subgroups $H$ of $G_1\times\cdots\times G_m$.
 \end{theorem}
An immediate consequence of Theorem \ref{dualitythm} is the inequality
 $$
 \left|\int_H f_1\otimes\cdots\otimes f_m \right|\leq\BL_{\mathbf{G}}(H,\mathbf{p})\min\left\{\prod_{j=1}^m\|f_j\|_{L^{p_j}(G_j)},\prod_{j=1}^m\|\widehat{f}_j\|_{L^{p_j'}(\widehat{G}_j)}\right\},$$
 which, provided $p_j<2$ for at least one $j$, constitutes a ``Hausdorff--Young improvement" of the standard discrete Brascamp--Lieb inequality given by \eqref{BLabstract}. 
 
\subsection*{Overview of the proof of Theorem \ref{dualitythm}}
In brief, our approach to Theorem \ref{dualitythm} amounts to establishing a good understanding of both $\BL_{\mathbf G}(H,\mathbf p)$  and $\BL_{\widehat{\mathbf G}} (H^\perp, \mathbf p')$, and connecting them via a discrete form of the Fourier--invariance property \eqref{Finv}.

For finitely-generated groups $\mathbf{G}$ (the subject of Theorem \ref{dualitythm}) the constant $\BL_{\mathbf G}(H,\mathbf p)$ is already quite well-understood -- see the forthcoming Theorem 
\ref{discrete} due to
Christ \cite{Ch1}. In particular, if $\BL_{\mathbf G}(H,\mathbf p)$ is finite, then it may be  achieved by testing on $f_j$ of the form $\chi_{\Gamma_j}$, where $\chi_{\Gamma_j}$ is the indicator function of a finite subgroup $\Gamma_j$ of $\pi_j(H)$; here $\pi_j:G_1\times\cdots\times G_m\rightarrow G_j$ is the $j$th coordinate map.

A similar result for $\BL_{\widehat{\mathbf G}} (H^\perp, \mathbf p')$ is, however, not readily available, and we establish this with the forthcoming Theorem \ref{torus}. In particular, we show that if $\BL_{\widehat{\mathbf G}} (H^\perp, \mathbf p')$ is finite, then it may be achieved by testing 
on $f_j$ of the form $\chi_{\Gamma_j}$ where now $\Gamma_j$ is a subgroup of $\widehat{\pi}_j(H^\perp)$ which contains its maximal torus; here $\widehat{\pi}_j$ denotes the $j$th coordinate map from $\widehat{G}_1\times\cdots\times \widehat{G}_m$ to $\widehat{G}_j$.

To connect these descriptions of $\BL_{\mathbf G}(H,\mathbf p)$ and $\BL_{\widehat{\mathbf G}} (H^\perp, \mathbf p')$, as Theorem \ref{dualitythm} requires, we appeal to a discrete form of the Fourier invariance property \eqref{Finv}, along with the fact that one of the two classes of indicator functions above is the Fourier transform of the other. So far, this approach is similar to that taken for Theorem \ref{basicthm} in the euclidean setting. However, Theorem \ref{dualitythm} presents difficulties beyond Theorem \ref{basicthm} due to the absence of a suitably strong form of Lieb's theorem in the discrete and compact settings. While it is true that we have access to suitable extremisers in both of those settings, these special functions $f_j$ alone are not sufficient to test for the \textit{finiteness} of the Brascamp--Lieb constant in either case (in contrast with the role of gaussians in the euclidean setting). As a result the equivalence of finiteness property \eqref{conj} (a qualitative form of Theorem \ref{dualitythm}) needs to be established before we can begin.

\subsection*{Structure of the paper}
In Section \ref{Sect:examples} we describe a variety of examples aimed at illustrating the action of the duality map \eqref{dualitymap}. In Section \ref{simple}, as preparation for our proof of Theorem \ref{dualitythm}, we establish \eqref{conj} in the \textit{euclidean} setting (a qualitative form of Theorem \ref{basicthm}) without appealing to Lieb's theorem. We also take the opportunity here to make some further structural observations about the duality map in the euclidean setting without recourse to Lieb's theorem.
In Section \ref{Sect:discrete} we apply the ideas from Section \ref{simple} to establish \eqref{conj} for finitely-generated groups (the aforementioned qualitative form of Theorem \ref{dualitythm}), which combined with Theorems \ref{discrete} and \ref{torus} leads to Theorem \ref{dualitythm}. We prove Theorem \ref{torus} in Section \ref{pf}.
Finally, in  Section \ref{further} we comment on some further lines of enquiry, including Brascamp--Lieb inequalities on non-abelian groups, and the duality map in the broader setting of Lorentz--Brascamp--Lieb inequalities.

\subsubsection*{Acknowledgements} We thank both Michael Cowling and Alessio Martini for helpful discussions.
The second author (E. Jeong) was  supported by the POSCO Science Fellowship and a KIAS Individual Grant no. MG070502.

\section{Brascamp--Lieb data and their duals: some examples}\label{Sect:examples}
Here we describe some examples of Brascamp--Lieb data $(H,\mathbf{p})$ and their duals $(H^\perp,\mathbf{p}')$, most of which may be naturally formulated in the general setting of LCA groups $G_1,\hdots, G_m$.
\subsection*{Fubini's theorem}
The very simplest example is when $H=G_1\times\cdots\times G_m$ and $p_1=\cdots=p_m=1$, where the inequality becomes an identity with $\BL(H,\mathbf{p})=1$, and merely amounts to Fubini's theorem. In this case $H^\perp=\{0\}$ and $p_j'=\infty$ for each $j$, and so the dual Brascamp--Lieb inequality amounts to the trivial statement that
$$
|f_1(0)\cdots f_m(0)|\leq \|f_1\|_\infty\cdots\|f_m\|_\infty.
$$
\subsection*{Young's convolution and H\"older's inequalities}
For proper subgroups $H$, identifying $\BL(H,\mathbf{p})$ is of course a much more delicate matter.
An important example is Young's convolution inequality on a LCA group $G$, which corresponds to the subgroup
\begin{equation}\label{dat1}
H=\{(x_1,\hdots,x_m)\in G\times \cdots\times G :x_1+\cdots+x_m=0\}
\end{equation} 
and exponents satisfying 
\begin{equation}\label{dat2}
\frac{1}{p_1}+\cdots+\frac{1}{p_m}=m-1,\;\;\;\;\;1\le p_1,\hdots,p_m\le \infty.
\end{equation}
In this case $$H^\perp=\{(x_1,\hdots,x_m)\in \widehat{G}\times \cdots\times \widehat{G} :x_1=\cdots=x_m\}$$ and $\frac{1}{p_1'}+\cdots+\frac{1}{p_m'}=1$, so that 
the dual data is that of the $m$-linear H\"older inequality and thus $\BL(H^\perp,\mathbf{p}')=1$. For discrete groups $G$ it is also well-known that $\BL(H,\mathbf{p})=1$, in accordance with Theorem \ref{dualitythm}. In fact much more is known in the case of Young's data (data satisfying \eqref{dat1} and \eqref{dat2}) -- indeed for quite general LCA groups $G$ the conjectural identity \eqref{dualityidgen} follows quickly from the work of Beckner \cite{Beckner}.

\subsection*{Higher order inner products}
Other interesting examples may be found in additive combinatorics in the form of bounds on higher order analogues of classical inner products. Following \cite{TaoHOFA}, for a locally compact abelian group $G$ and natural number $k$, the \emph{Gowers inner product} $\mathfrak G_k((f_\omega)_{\omega\in \{0,1\}^k})$ on $G$ is defined by
\[ \mathfrak G_k\big((f_\omega)_{\omega\in \{0,1\}^k}\big) =\int_{G^{k+1}} \prod_{\omega\in\{0,1\}^k} \mathcal C^{|\omega|} f_\omega(x+\omega \cdot h) \,d\mu(x)\,d\mu(h_1)\cdots d\mu(h_k),\]
where $\omega=(\omega^1,\cdots,\omega^k)\in\{0,1\}^k$, $|\omega|=\omega^1+\cdots +\omega^k$,  $\mathcal Cf=\bar f$ is the complex conjugate map, and $\mu$ is the normalised Haar measure on $G$. This gives rise to the   celebrated \emph{uniformity norm}  $ \|f\|_{U^k(G)}:= \mathfrak G_k((f_\omega)_{\omega\in \{0,1\}^k})^{1/2^k}$ by setting $f_\om=f $  for all $\om$.  The Gowers inner product is an example of a Brascamp--Lieb form \eqref{BLform} corresponding to the group $H=H_k(G)$ given  by
\[ H_k(G):=\left\{ \mathbf y =(y_\om)\in G^{\{0,1\}^k} : \,  y_\om =x+\om\cdot h, \quad x\in G,\, h\in G^k\right\}.\]
The uniformity norms were first introduced by Gowers \cite{Go} to study Szemer\'edi's theorem on a finite abelian group, and the notion was extended to locally compact abelian groups by Eisner and Tao \cite{ET}. Interpreting the dual data $(H_k(G)^\perp, \mathbf{p}')$ is also straightforward in this setting, indeed
$$H_k(G)^\perp=\Big\{  \textbf x=(x_\omega) \in \widehat{G}^{\{0,1\}^k} :\, \sum_{\omega\in \mathbf F}  x_\omega =0\text{ for any  face $\mathbf F$ of $\{0,1\}^k$}\Big\}.$$
This is manifestly isomorphic to $$\HK^{k}(\widehat{G}, \le k-2):=\Big\{  \textbf x=(x_\omega) \in \widehat{G}^{\{0,1\}^k} :\, \sum_{\omega\in \mathbf F}  (-1)^{|\om|}x_\omega =0\text{ for any  face $\mathbf F$ of $\{0,1\}^k$}\Big\},$$ known as the \textit{Host--Kra group} (see \cite{TaoHOFA}) of $\widehat{G}$. 
In the case of finitely-generated groups $G$, Theorem \ref{dualitythm} therefore establishes that 
\begin{equation}\label{hk}
\BL_{\mathbf G}(H_k(G),\mathbf p)=\BL_{\widehat{\mathbf G}}(\HK^{k}(\widehat{G},\le k-2),\mathbf p').
\end{equation}

\subsubsection*{Remark} The identity \eqref{hk} and its analogue for euclidean spaces provided by Theorem \ref{basicthm} offer an alternative perspective on the question of Lebesgue space bounds for the Gowers inner products. Here the literature appears to be rather incomplete. In \cite{ET} the authors showed that $\BL_{\mathbf G}(H_k(G),\mathbf p)<\infty$ when $p_\om =p_k:=2^k/(k+1)$,\footnote{In particular they showed that  $\BL_{\mathbf G}(H_k(\R),\mathbf p) = 2^k/ (k+1)^{(k+1)/2}$ for $p_\om=p_k$.} and certain off-diagonal Lebesgue bounds were later established by Neuman \cite{N}.
We note that if $p_\om\geq 2$ for each $\om$  then the \emph{inequality} 
$\BL_{\widehat{\mathbf G}}(\HK^{k}(\widehat{G},\le k-2),\mathbf p')\leq \BL_{\mathbf G}(H_k(G),\mathbf p)$ is an elementary consequence of the abstract form of the Fourier invariance property \eqref{Finv} (see \eqref{poisson}) followed by the Hausdorff--Young inequality. Theorem \ref{dualitythm} therefore reduces the task of bounding $\BL_{\mathbf G}(H_k(G),\mathbf p)$ to the (ostensibly simpler if $p_\om\geq 2$ for all $\om$) task of bounding $\BL_{\widehat{\mathbf G}}(\HK^{k}(\widehat{G},\le k-2),\mathbf p')$.

We refer to  \cite{BBBCF}  for some further examples of Brascamp--Lieb data and their duals arising in euclidean harmonic analysis and PDE.

\section{Structural properties of the duality map: elementary arguments in the euclidean setting}\label{simple}
The main purpose of this section is to provide an elementary proof (in particular, avoiding Lieb's theorem) of the equivalence of finiteness of the euclidean Brascamp--Lieb constant under the duality map \eqref{dualitymap}. As indicated in the introduction, this may be viewed as a preparatory step towards the more abstract framework of Section \ref{Sect:discrete}, where suitably strong analogues of Lieb's theorem are unavailable. Throughout this section $H_1,\hdots,H_m$ and $H\leq H_1\times\cdots\times H_m$ denote euclidean spaces with Lebesgue measure, and $\mathbf{p}=(p_j)\in [1,\infty]^m$.
\begin{theorem}\label{hilbert}
$$\BL(H,\mathbf{p})<\infty\iff\BL(H^\perp,\mathbf{p}')<\infty$$
\end{theorem}
We stress that Theorem \ref{hilbert}, which is a qualitative version of Theorem \ref{basicthm}, follows relatively quickly if one is prepared to appeal to Lieb's theorem, as is apparent on inspecting the short proof of Theorem \ref{basicthm} in \cite{BBBCF}.
Our elementary proof of Theorem \ref{hilbert} relies on a finiteness characterisation of the Brascamp--Lieb constant $\BL(H,\mathbf p)$ established by Carbery, Christ, Tao and the first author in \cite{BCCT1,BCCT2}; see also the earlier work of Carlen, Lieb and Loss \cite{CLL} in the so-called rank-1 case (where $\dim H_j=1$ for all $j$). This states that $\BL(H,\mathbf p)<\infty$ if and only if the scaling and dimension conditions
\begin{align}
\dim H &=\sum_{j=1}^m \frac{\dim H_j}{p_j},\label{scale}\\
\dim V &\le \sum_{j=1}^m \frac{\dim(\pi_j V)}{p_j}\label{dim}
\end{align}
hold for all subspaces $V$ of $H$.
Here $\pi_j:H_1\times\cdots\times H_m\rightarrow H_j$ is the $j$th coordinate projection.
Importantly, the proof of this finiteness criterion provided in \cite{BCCT2} is elementary, involving only linear algebraic identities and induction on the dimension of $H$ -- in particular, it does not appeal to phenomena firmly tied to the underlying euclidean structure, such as Lieb's theorem.  
\begin{proof}[Proof of Theorem \ref{hilbert}] Let $n=\dim H$ and  $n_j =\dim H_j$ for  $1\le j\le m.$  Since $(H^\perp)^\perp=H$ and $(\mathbf{p}')'=\mathbf{p}$, it is enough to prove that $(H^\perp,\mathbf p')$ obeys the scaling and dimension conditions \eqref{scale} and \eqref{dim} under the assumption that $\BL(H, \mathbf p)<\infty$.
The scaling condition \eqref{scale} for $(H^\perp ,\mathbf p')$ is immediate from that for $(H,\mathbf{p})$ since
\begin{equation}\label{scaleperp}\dim H^\perp =\sum_{j=1}^m n_j -\dim H  =\sum_{j=1}^m \Big(n_j -\frac{n_j}{p_j}\Big)=\sum_{j=1}^m \frac{n_j}{p_j'}. \end{equation}
In order to establish \eqref{dim} for $(H^\perp,\mathbf p')$, 
let $W$ be an arbitrary subspace of $H^\perp$. By \eqref{scaleperp}, the dimension condition \eqref{dim} for the datum $(H^\perp,\mathbf p')$ is equivalent to 
\begin{equation}\label{red}
\sum_{j=1}^m \dim (\pi_j W)^\perp -\dim W^\perp \le \sum_{j=1}^m \frac{\dim (\pi_j W)^\perp}{p_j}.
\end{equation}
Here $\pi_j=\pi_j|_{H^\perp}$, and $W^\perp$ and $(\pi_j W)^\perp$ are the orthogonal complements of $W$ and $\pi_j W$, respectively; that is, $W\oplus W^\perp =H^\perp$ and $(\pi_j W)\oplus (\pi_j W)^\perp =H_j$. To conclude it therefore suffices to construct a subspace $V$ of $H$ satisfying
\begin{align}
\dim V &\ge \,\sum_{j=1}^m \dim (\pi_j W)^\perp -\dim W^\perp, \label{glo}\\
\dim (\pi_j V )&\le\, \dim (\pi_j W)^\perp, \quad j=1,\hdots,m, \label{loc}
\end{align}
since $(H,\mathbf p)$ obeys \eqref{dim} by the finiteness of $\BL(H,\mathbf p)$.
These considerations naturally prompt the choice 
\[V =(\pi_1 W)^\perp \times \cdots \times (\pi_m W)^\perp \cap H.\]
Clearly  $V$ is  a subspace of $H$ and satisfies the relation \eqref{loc}. It is also straightforward to show that $V$ satisfies \eqref{glo}, since  $U:= (\pi_1 W)^\perp \times \cdots \times (\pi_m W)^\perp $ is a subspace of $W^\perp \oplus H$, the orthogonal complement of $W$ in the whole space $\widetilde H := H_1\times  \cdots\times H_m$.
Indeed, for $u\in U$ and $w\in W$,
$ \langle u, w\rangle_{\widetilde H} =\sum_{j=1}^m \langle \pi_ju, \pi_jw\rangle_{H_j} =0,
$
since $\pi_j u\in (\pi_j W)^\perp$.  Thus we have
\begin{align*}
\dim V 
&= \dim U +\dim H -\dim (U+H)\\
&\ge \dim U +\dim H -\dim W^\perp -\dim H\\
&=\sum_{j=1}^m \dim (\pi_j W)^\perp -\dim W^\perp,
\end{align*} 
as claimed. Here the inequality is due to the elementary fact that
$U+H \le W^\perp \oplus H$.
\end{proof}
There are other structural properties of Brascamp--Lieb data that are natural to investigate under the duality map \eqref{dualitymap}. For example,
a datum  $(H,\mathbf p)$ is referred to as \emph{simple} if there is no non-trivial proper subspace $V$ of $H$ such that \eqref{dim} holds with equality. If $(H,\mathbf p)$ is simple, then it was established in \cite{BCCT1} (again, see also \cite{CLL}) that it has a unique gaussian extremiser (modulo elementary scalings), provided all of $p_j$ are finite.
\begin{proposition}\label{simp} If $\mathbf{p}\in (1,\infty)^m$ then
simplicity is preserved under the duality map,  that is, $(H,\mathbf{p})$ is simple if and only if $(H^\perp,\mathbf{p}')$ is simple.
\end{proposition}

\begin{proof}
Assume that $(H, \mathbf p)$ is simple. From the proof of Theorem \ref{hilbert}, it is enough to show that for any non-zero proper subspace $W$ of $H^\perp$ there is strict inequality in \eqref{red}; that is,
\begin{equation}\label{s_red}
\sum_{j=1}^m \dim (\pi_j W)^\perp -\dim W^\perp < \sum_{j=1}^m\frac{\dim (\pi_j W)^\perp}{p_j}.
\end{equation}
Now, we have already shown that there is a subspace $V$ of $H$ such that \eqref{glo} and \eqref{loc} hold, and so if this subspace is non-zero and proper, then \eqref{s_red} is evident from the simplicity of $(H, \mathbf p)$. Therefore it is enough to deal with the hypothetical possibilities that $V=\{0\}$ or $H$.
For the case $V=\{0\}$ we consider a dichotomy.
If $\dim (\pi_j V) < \dim (\pi_j W)^\perp$ for some $j$, then \eqref{s_red} is immediate as $1/p_j\neq 0$.
Alternatively, if $\dim (\pi_j V ) = \dim (\pi_j W)^\perp$ for all $j$,  then the right-hand side of \eqref{glo} equals $-\dim W^\perp$, which is nonzero by assumption. Thus equality  in \eqref{glo} is impossible, which yields  \eqref{s_red}.
For the case $V=H$, by \eqref{loc} and the surjectivity of $\pi_j|_{H}$  which follows from \eqref{scale} and \eqref{dim}, we have
\[ n_j =\dim (\pi_j V) \le \dim (\pi_j W)^\perp \le n_j,\]
and $\pi_j W=\{0\}$ for all $j$. This contradicts the assumption that  $W \neq \{0\}$, and so $V=H$ cannot occur.
\end{proof}

\subsubsection*{Remark} There are Brascamp--Lieb data  $(H,\mathbf p)$ with $p_j=\infty$ for some $j$ such that $(H,\mathbf p)$ is simple but $(H^\perp,\mathbf p')$ is not. For example if $H=\{(x,0, -x): x\in \R\}\subset \R^2\times \R$, $p_1=\infty$ and $p_2=1$. Then $(H,\mathbf p)$ is simple and its dual data $(H^\perp,\mathbf p')$ is given by
$ H^\perp = \langle\{(1,0,1), (0,1,0)\}\rangle \subset \R^2\times \R.$
Note that for $W= \langle\{(1,0,1)\}\rangle$ it follows that $\dim W =\dim \pi_1(W)/p_1' +\dim \pi_2(W)/p_2'$, and hence $(H^\perp,\mathbf p')$ is not simple.

Of course statements about the duality map \eqref{dualitymap} involving gaussian-extremisability, or gaussian near-extremisability (such as Lieb's theorem itself) are strongly tied to the  context of euclidean spaces, and so are less relevant for the purposes of this paper. However, for completeness we provide such a statement whose proof, which appeals to Lieb's theorem, is implicit in \cite{BBBCF}.
\begin{proposition}\label{extpro}
If $(f_j)$ is a gaussian extremiser for $(H,\mathbf{p})$ then $(\widehat{f}_j)$ is a gaussian extremiser for $(H^\perp,\mathbf{p}')$.
\end{proposition}
It should be remarked that the Fourier transform does not in general map \emph{all} extremisers for $(H,\mathbf{p})$ to extremisers for $(H^\perp,\mathbf{p}')$. One simply has to look to Fubini's theorem for an example -- see Section \ref{Sect:examples}.

\section{Proof of Theorem \ref{dualitythm}}\label{Sect:discrete}
We begin by establishing suitable finiteness criteria for the Brascamp--Lieb constants  when the underlying groups are either finitely-generated or the duals of such. From there the elementary argument used in Section \ref{simple} may be adapted to first establish a qualitative form of Theorem \ref{dualitythm}, referring only to the invariance of \emph{finiteness} of the Brascamp--Lieb constant under the duality map (a discrete analogue of Theorem \ref{hilbert}). Upgrading this qualitative statement to Theorem \ref{dualitythm} then relies on an analogue of the Fourier-invariance property \eqref{Finv}
for LCA groups.

Throughout the section, the $G_j$ denote finitely-generated abelian groups, as in the statement of Theorem \ref{dualitythm}. Consequently, $G_j\cong \mathbb Z^{n_j}\times F_j$ and $\widehat{G}_j \cong \T^{n_j}\times  \widehat F_j$ for a nonnegative integer $n_j$ and a finite abelian group $F_j$. Here $\T^{n_j}\cong \R^{n_j}/\Z^{n_j}$ is the $n_j$-dimensional flat torus. If $H$ is a subgroup of $G_1\times \cdots \times G_m$, then $H$ is also isomorphic to $\mathbb Z^n\times F$ for some nonnegative integer $n$ and a finite abelian group $F$. We refer to $n$ as the rank of $H$. 
We denote by $\pi_j : H\to G_j$ the $j$th component map,\footnote{Strictly speaking this is the $j$th component map $\pi_j$ restricted to $H$. We gloss over this distinction for the sake of notational simplicity.}  and following the usual conventions in this context, we equip $G_j$ and $H$ with counting measures $\mu_{G_j}$ and $\mu_H$, respectively.

An appropriate finiteness criterion for $\BL_{\mathbf G}(H,\mathbf p)$ was established by Carbery, Christ, Tao and the first author in \cite{BCCT2}. 
Under this finiteness condition an explicit expression for $\BL_{\mathbf G}(H,\mathbf p)$ was then obtained by Christ, Demmel, Knight, Scanlon and Yelick  \cite{CDKSY} and Christ \cite{Ch1}. These are summarised in the following theorem.
\begin{theorem}[\cite{BCCT2}, \cite{CDKSY}, \cite{Ch1}]\label{discrete}
Let $H$ be a  subgroup of $G_1\times \cdots \times G_m$. For $1\le p_j\le \infty$, $\BL_{\mathbf G}(H,\mathbf p)$ is finite if and only if
\begin{equation}\label{rank}
\rank V\le \sum_{j=1}^m  p_j^{-1} \rank \pi_j(V)
\end{equation}
for every  subgroup $V$ of $H$.  In addition, if $\BL_{\mathbf G} (H,\mathbf p)$ is finite,
\begin{equation}\label{optimal}
 \BL_{\mathbf G} (H, \mathbf p) =\max_{V\leq H} |V| \prod_{j=1}^m |\pi_j (V)|^{-1/p_j},
 \end{equation}
where the maximum is taken over all finite subgroups $V$ of $H$.
 \end{theorem}
 Of course the finite subgroups in \eqref{optimal} are the subgroups of the torsion subgroup of $H$, and so if $H$ is torsion-free then $\BL_{\mathbf G}(H,\mathbf p)$ may only take the values $1$ and $\infty$.

We now turn our attention to a similar theorem in the context of the duals of finitely-generated abelian groups.
Let  $S$ be a  subgroup of $\widehat G_1\times\cdots \times \widehat G_m$ and let $\widehat\pi_j : S\to \widehat G_j$ be the $j$th component map. Again, as is standard in this context, we equip $S$ and $\widehat G_j$ with Haar measures $\mu_{S}$ and $\mu_{\widehat G_j}$, respectively, so that  $\mu_{S}(S)=\mu_{\widehat G_j}(\widehat G_j)=1$.
It is reasonable to expect that a sufficient condition for the finiteness of $\BL_{\widehat {\mathbf G}}(S,\mathbf p)$ is similar to that for the so-called \emph{localised Brascamp--Lieb constant} established in \cite[Theorem 2.2]{BCCT2} in the euclidean setting, and this is indeed the case. 
In what follows $\mathscr F:=S/ \fm(S)$ and $\mathscr F_j:=\widehat{G}_j/\fm(\widehat G_j)$ are the quotients of $S$ and $\widehat G_j$ by their maximal tori $\fm(S)$ and $\fm(\widehat G_j)$ respectively, and $\phi_j : \mathscr F \to\mathscr F_j$ is the homomorphism induced by $\widehat \pi_j:S\rightarrow\widehat{G}_j$ (that is, $\phi_j(x+\fm(S))=\widehat \pi_j(x) + \fm(\widehat G_j)$).  
\begin{theorem}\label{torus}
Let $S$ be a closed subgroup of $\widehat G_1\times \cdots \times \widehat G_m$.  Then for $1\le p_j\le \infty$, $\BL_{\widehat{\mathbf G}}(S,\mathbf p)$ is finite if and only if  for every closed subgroup $W$ of $S$
\begin{equation}\label{codim}
\codim_{S} (W)\ge \sum_{j=1}^m  p_j^{-1} \codim_{\widehat G_j} (\widehat \pi_j(W)).
\end{equation}
In addition, if $\BL_{\widehat{\mathbf G}}(S,\mathbf p)$ is finite, then 
\begin{equation}\label{optimal_t}
 \BL_{\widehat{\mathbf G}}(S, \mathbf p)= \max_{\Gamma \leq \mathscr F} |\mathscr F|^{-1}|\Gamma | \prod_{j=1}^m \Big(|\mathscr F_j|^{-1} |\phi_j(\Gamma)|\Big)^{-1/p_j}\end{equation}
where the maximum is taken over all subgroups $\Gamma$ of $\mathscr F$. 
\end{theorem}
The expressions \eqref{optimal} and \eqref{optimal_t} are of course structurally very similar. We clarify that the cardinalities in \eqref{optimal_t} are the Haar measures of subsets of the finite groups $\mathscr{F}$ and $\mathscr{F}_j$ (thought of as compact groups and so given total mass $1$), while the cardinalities in \eqref{optimal} are Haar measures of the appropriate finite groups (thought of instead as discrete groups, and so assigned counting measures).

Theorem \ref{torus} complements work of Bramati \cite{Brami}, who studied the Brascamp--Lieb inequality on tori as a special case of the Brascamp--Lieb inequality on compact homogeneous spaces.
Postponing the proof of Theorem \ref{torus} until Section \ref{pf},
it remains to use Theorems \ref{discrete} and \ref{torus}  to prove Theorem \ref{dualitythm}.  We begin by establishing some elementary lemmas.
\begin{lemma}\label{rank_dim} Let $G$ be a finitely-generated abelian group.
\begin{itemize}
\item[(i)] If $V$ is a subgroup of $G$ then
\begin{equation}\label{rd}
 \dim V^\perp =\rank G -\rank V;
 \end{equation}
\item[(ii)] if $W$ is a subgroup of $\widehat G$, the dual group of $G$, then
\begin{equation}\label{dr}
\rank W^\perp \le \dim \widehat{G} -\dim W,
\end{equation}
with equality if $W$ is closed.
\end{itemize}
\end{lemma}
\begin{proof}
Beginning with (i), since $V$ is a closed subgroup of $G$, we have $\widehat V\cong \widehat G/V^\perp$. 
Hence by the quotient manifold theorem \cite[Theorem 7.10]{JLee},
\[ \rank V=\dim\widehat V= \dim \widehat G-\dim V^\perp= \rank G-\dim V^\perp. \]
Statement (ii) follows from (i) and the Pontryagin duality theorem.
Indeed, for any closed subgroup $W$ of $\widehat {G}$, equality in \eqref{dr} follows from \eqref{rd} and $(W^\perp)^\perp =W$.  In general, the inequality \eqref{dr} follows by applying this identity to $\overline W$, the closure of $W$.
\end{proof}
\begin{lemma}\label{trivial} Let $G_j$ be a LCA group for each $1\le j\le m$. For any subgroup $V$ of $G_1\times \cdots \times G_m$,
\[ (\pi_1 V)^\perp \times \cdots \times (\pi_m V)^\perp \subset V^\perp.\]
Here, as usual, $\pi_j$ is the projection of $V$ onto the $j$th component $G_j$. 
\end{lemma}
\begin{proof}  Set $G=G_1\times \cdots \times G_m$, and for $x\in G$ and $\gamma \in \widehat G$ write
$x=(x_1,\cdots, x_m)$  and $\gamma =(\gamma_1,\cdots, \gamma_m)$ with  $x_j\in G_j$ and $\gamma\in \widehat G_j$. If $\gamma \in (\pi_1 V)^\perp \times  \cdots \times (\pi_m V)^\perp$ then
\[ \gamma(x) =\prod_{j=1}^m\gamma_j(x_j) =1 \quad \text{for any $x=(x_1,\cdots,x_m)\in V$},\]
since $\gamma_j(x_j)=1$ for any $x_j\in \pi_j(V)$. 
Hence $\gamma\in V^\perp$.
\end{proof}

\begin{proof}[Proof of Theorem \ref{dualitythm}]
We first verify that 
\begin{equation}\label{qual}
\BL_{\mathbf G}(H, \mathbf p)<\infty\implies \BL_{\widehat{\mathbf G}}(H^\perp, \mathbf p')<\infty.
\end{equation} 
By Theorem \ref{torus} it is enough to prove \eqref{codim} for any closed subgroup $W$ of $H^\perp$; equivalently
\begin{equation}\label{eqdim}
\sum_{j=1}^m   p_j^{-1} \codim_{\widehat{G}_j} \widehat \pi_j(W) \ge  \sum_{j=1}^m   \codim_{\widehat{G}_j} \widehat \pi_j(W) -\codim_{H^\perp} W.
\end{equation}
For such a $W$ we set
\[V =(\widehat \pi_1 W)^\perp \times \cdots \times (\widehat \pi_m W)^\perp \cap H.\]
Clearly $V$ is  a subgroup of $H$, hence by \eqref{rank},
\begin{equation}\label{tem_r}
\rank V\le \sum_{j=1}^m  p_j^{-1} \rank \pi_j(V).
\end{equation}
Since $\pi_j V\subset (\widehat \pi_j W)^\perp$, by Lemma \ref{rank_dim} we see that
\[ \rank \pi_j V\le \dim \widehat{G}_j- \dim \widehat \pi_j W  =\codim_{\widehat{G}_j} \widehat \pi_j W, \]
and so
$$
\sum_{j=1}^m  p_j^{-1} \rank \pi_j(V)\leq \sum_{j=1}^m   p_j^{-1} \codim_{\widehat{G}_j} \widehat \pi_j(W).
$$
In order to conclude \eqref{eqdim} it therefore suffices to show that 
$$
\rank V\geq \sum_{j=1}^m   \codim_{\widehat{G}_j} \widehat \pi_j(W) -\codim_{H^\perp} W.
$$
To do this we first write
\[ \rank V =\rank U +\rank H -\rank (U +H),\]
where $U=(\widehat \pi_1 W)^\perp \times \cdots \times (\widehat \pi_m W)^\perp.$
Notice that $(U +H) \subset W^\perp$ by Lemma \ref{trivial} and  the assumption that $W\subset H^\perp$. From this we have
\begin{align*}
\rank V  &\ge  \sum_{j=1}^m \rank(\widehat \pi_j W)^\perp +\rank H -\rank W^\perp\\
&= \sum_{j=1}^m \codim_{\widehat{G}_j} \widehat \pi_j W -\codim_{H^\perp} W\\
&\quad  +\dim H^\perp -\dim W +\rank H -\rank W^\perp\\
&= \sum_{j=1}^m \codim_{\widehat{G}_j}\widehat  \pi_j W -\codim_{H^\perp} W,
\end{align*}
as required. Here we also used that $\dim H^\perp -\dim W +\rank H -\rank W^\perp=0$, which follows from Lemma \ref{rank_dim}.

We now establish the converse of the implication \eqref{qual} by very similar reasoning. By Theorem \ref{discrete} it suffices to show that \eqref{rank}, or equivalently
\begin{equation}\label{tired}
\sum_{j=1}^m \frac{ \rank \pi_jV}{p_j'} \le \sum_{j=1}^m   \rank \pi_j V -\rank V,
\end{equation}
holds for all $V\leq H$.
As before, we set 
\[W =(\pi_1 V)^\perp \times \cdots \times (\pi_m V)^\perp \cap H^\perp.\]
Since $W$ is a closed subgroup of $H^\perp$, by Theorem \ref{torus} we have
\begin{equation}\label{tem_d}
\codim_{H^\perp} W\ge \sum_{j=1}^m  \frac1{p_j'}\codim_{\widehat{G}_j} \widehat \pi_jW.
\end{equation}
By Lemma \ref{rank_dim}, we have $$\dim \widehat \pi_j W \le \dim (\pi_j V)^\perp =\rank G_j -\rank \pi_j V,$$ and so
\[\sum_{j=1}^m  \frac1{p_j'}\codim_{\widehat{G}_j} \widehat \pi_jW=\sum_{j=1}^m  \frac{ 1}{p_j'}(\dim \widehat{G}_j-\dim \widehat  \pi_jW)\ge \sum_{j=1}^m  \frac{ \rank \pi_j V}{p_j'}.\]
Moreover, using Lemmas \ref{trivial} and  \ref{rank_dim}, we obtain
\begin{align*}
\codim_{H^\perp} W&= \dim H^\perp -\dim W\\
&\le \dim \big(H^\perp +((\pi_1 V)^\perp \times \cdots \times (\pi_m V)^\perp)\big)-   \dim \big((\pi_1 V)^\perp \times \cdots \times (\pi_m V)^\perp\big) \\
&\le  \dim V^\perp-\sum_{j=1}^m\dim (\pi_j V)^\perp 
=\sum_{j=1}^m \rank \pi_j V -\rank V.
\end{align*} 
The required inequality \eqref{tired} now follows from \eqref{tem_d} and the above two inequalities.

Now we have established the equivalence of the finiteness of $\BL_{\mathbf G}(H,\mathbf p)$ and $\BL_{\widehat{\mathbf G}}(H^\perp,\mathbf p')$, we may appeal to the expressions \eqref{optimal} and \eqref{optimal_t} and complete the proof of Theorem \ref{dualitythm}.
We assume, as we may, that $G_j= \Z^{n_j}\times F_j$ for some finite abelian group $F_j$, so that $\widehat G_j =\T^{n_j}\times \widehat F_j$.   This will suffice as all integral expressions involved in the definitions of $\BL_{\mathbf G}(H,\mathbf p)$ and $\BL_{\widehat{\mathbf G}}(H^\perp,\mathbf p')$ are invariant under taking compositions with the natural group homomorphisms.

It remains to establish \eqref{isome} assuming  that $\BL_{\mathbf G}(H,\mathbf p)$, or equivalently $\BL_{\widehat{\mathbf G}}(H^\perp,\mathbf p')$, is finite. This will be a direct consequence of an abstract form of the Fourier-invariance property \eqref{Finv}, along with the fact that there exist extremisers $\mathbf f=(f_j)$ and $\mathbf g=(g_j)$ for the data $(H,\mathbf p)$ and $(H^\perp,\mathbf p')$, respectively, taking the form  
\begin{equation}\label{form}
 f_j=\chi_{\{0\}}\otimes \chi_{A_j}\quad \text{and}\quad g_j =\chi_{\T^{n_j}}\otimes \chi_{B_j}
 \end{equation}
where $A_j$ and $B_j$ are subgroups of $F_j$ and $\widehat F_j$, respectively. Here we are appealing to both \eqref{optimal} and \eqref{optimal_t}.

Let  $f_j:= |A_j|^{-1/p_j}\chi_{\{0\}}\otimes \chi_{A_j}$ be such an extremiser for the data $(H,\mathbf p)$, that is 
\begin{equation}\label{extremiser0}
\BL_{\mathbf G}(H,\mathbf p)=\int_H f_1\otimes \cdots \otimes f_m(x) d\mu_H(x).
\end{equation}
Next we appeal to the abstract Fourier--invariance property (see for example \cite[Theorem 3.6.3]{DE})
\begin{equation}\label{poisson}
\int_H f_1\otimes \cdots \otimes f_m(x) d\mu_H(x)= \int_{H^\perp}   \widehat f_1\otimes \cdots \otimes  \widehat f_m (\gamma)d \mu_{H^\perp}(\gamma),
\end{equation}
where the Fourier transform $\widehat f_j$ is given by $$\widehat f_j(\gamma_j) =\int_{G_j} f_j(x_j) \overline{\gamma_j(x_j)} d\mu_{G_j}(x_j).$$ Standard reasoning reveals that $\widehat f_j =|A_j|^{1/p_j'}\chi_{\T^{n_j}}\otimes \chi_{A_j^\perp}$ and $\|\widehat f_j\|_{L^{p_j'}(G_j)} 
=1$, and so
\[ \int_{H^\perp}   \widehat f_1\otimes \cdots \otimes  \widehat f_m (\gamma)d\mu_{H^\perp}(\gamma)\le \BL_{\widehat {\mathbf G}}(H^\perp,\mathbf p'). \]
Combining this with \eqref{extremiser0} and \eqref{poisson}, we conclude that $\BL_{\mathbf G}(H,\mathbf p)\le \BL_{\widehat {\mathbf G}}(H^\perp,\mathbf p')$. A very similar argument, involving extremisers for $(H^\perp, \mathbf p')$ of the form $g_j=\chi_{\T^{n_j}}\otimes \chi_{B_j}$, leads to the reverse inequality $\BL_{\widehat {\mathbf G}}(H^\perp,\mathbf p')\le \BL_{\mathbf G}(H,\mathbf p)$. We leave this to the reader.
\end{proof}

\subsubsection*{Remark} As we have seen, if the underlying abelian groups are  finitely-generated, or the duals of such, an extremiser always exists whenever the Brascamp--Lieb constant is finite. In particular,  the Fourier transform maps extremisers of the form  $f_j=\chi_{\Phi_j(\{0\}\times A_j)}$ for $(H,\mathbf p)$ to extremisers of the form $g_j=\chi_{\Psi_j(\T^{n_j} \times B_j)}$ for $(H^\perp, \mathbf p')$; here $\Phi_j:\Z^{n_j}\times F_j\to G_j$ (resp. $\Psi : \T^{n_j}\times \widehat F_j\to \widehat G_j$) is a homomorphism and $A_j$ (resp. $B_j$) is a subgroup of $F_j$ (resp. $\widehat F_j$). Notice that, when $f_j$  is an extremiser for  $(H,\mathbf p)$ with $\|f_j\|_{L^{p_j}(G_j)}$=1, 
\[
\BL_{\mathbf G}(H,\mathbf p)
= \int_{H^\perp} \widehat f_1\otimes \cdots\otimes \widehat f_m d\mu_{H^\perp}\le \BL_{\mathbf {\widehat G}}(H^\perp,\mathbf p'),
\]
by the Fourier-invariance property  \eqref{poisson}, and that equality must hold here
thanks to \eqref{dualityidgen}. Consequently $\widehat f_j$ is an extremiser for $(H^\perp, \mathbf p').$

\section{Proof of Theorem \ref{torus}}\label{pf}
Our proof of Theorem \ref{torus} will follow the induction argument of \cite{BCCT2}. 
Since the dual of a finitely-generated abelian group is isomorphic to the product of a torus and a finite group, Theorem \ref{torus} may be restated as follows.
\begin{theorem}\label{re_torus}
Suppose that $G\cong \T^n\times F$ and $G_j\cong \T^{n_j}\times F_j$ for some finite groups $F_j$, $F$, and integers $n_j, n\ge 0$. Let $\mu_G$ and $\mu_{G_j}$ be the Haar probability measures of $G$ and $G_j$, respectively. Let $\phi_j :G\to G_j$ be a group homomorphism for each $j$. Then  for $1\le p_j\le \infty$ there is a finite constant $C>0$ such that
\begin{equation}\label{BLtorus}
\int_{G}\prod_{j=1}^m f_j\circ \phi_j (x)\, d\mu_G(x) \le C \prod_{j=1}^m \|f_j\|_{L^{p_j} (G_j)}
\end{equation}
if and only if   
\begin{equation}\label{dim'}
\codim_{G} ( W)\ge \sum_{j=1}^m  p_j^{-1} \codim_{G_j} (\phi_j( W))
\end{equation}
for every closed subgroup $W$ of $G$.
In addition, if \eqref{BLtorus} holds with a finite constant, then the smallest constant is given by
\[ \max_{W\subset F} |F|^{-1}|W| \prod_{j=1}^m \Big(|F_j| |\Phi_j^2(0,W)|^{-1}\Big)^{1/p_j},\] 
where the maximum is taken over all subgroups $W$ of $F$. 
Here $\Phi_j=(\Phi_j^1,\Phi_j^2): \T^n\times F \cong G\xrightarrow[]{\phi_j} G_j\cong  \T^{n_j}\times F_j$ coincides with $\phi_j$ up to isomorphism.
\end{theorem}
We note that since Theorem \ref{re_torus} is intended to be applied to the dual of a finitely-generated group, its choice of Haar measure normalisation is consistent with the conventions discussed in the introduction.
\begin{proof}
By the algebraic invariance of the statement of Theorem \ref{re_torus} it suffices to assume that $G=\T^n\times F$ and $G_j= \T^{n_j}\times F_j$ for each $j$.

We first prove the necessity of \eqref{dim'} using  a Knapp type example.
Let $W$ be a nontrivial subgroup of $G$, $0<r\ll 1$ and let $f_j$ be the characteristic function of $\{x_j\in G_j : \dist(x_j, \phi_j(w))\le r\,\,\text{for some } w\in W\}$. Then
\[\|f_j\|_{L^{p_j}(G_j)}^{p_j} \lesssim  r^{\dim G_j -\dim \overline{\phi_j(W)}}=r^{\codim_{G_j} (\phi_j(\overline W))}.\]
Since $\prod_{j=1}^m f_j\circ \phi_j (x) =1$ whenever $\dist(x,w)\le cr$, for some $w\in \overline W$ and constant $c>0$, 
\[\int_{G}\prod_{j=1}^m f_j\circ \phi_j (x) d\mu_G(x)\gtrsim r^{\dim G-\dim \overline W}. \]
Inequality \eqref{dim'} follows on letting $r\to0$.\\

We now turn to the sufficiency of \eqref{dim'}. Without loss of generality we may assume that $1\le p_j<\infty$ for all $j=1,\hdots,m$, since the terms $j$ with $p_j=\infty$ have no effect on the estimate \eqref{BLtorus}.  Observe first that, by \eqref{dim'} applied with $W=G$,  we have $\dim \phi_j (G)=\dim G_j$ for all $j$.\\

We first deal with the case that $G$ and the $G_j$ are  connected, that is $G=\T^n$ and $G_j=\T^{n_j}$, and argue by induction on $n$.\\

If $n=1$ then clearly $\dim G_j\in\{0,1\}$ for each $j$. Neglecting, as we may, those factors for which $\dim G_j=0$, and noting that $\|f_j\circ \phi_j\|_{L^p(G)}=\|f_j \|_{L^p(G_j)}$ for all $p$ otherwise, the estimate \eqref{BLtorus} (with constant $C=1$) follows quickly from \eqref{dim'} by H\"older's inequality.

We now fix $n>1$ and  assume that the estimate \eqref{BLtorus} holds with $C=1$ under the assumption \eqref{dim'}, whenever $\dim G<n$.  We first deal with the case that  there is a non-trivial connected closed subgroup $W\subsetneq G$ such that $\codim_G(W)=\sum_{j=1}^{m} p_j^{-1}\codim_{G_j }(\phi_j(W))$. We call such $W$ a critical subgroup of $G$. Since $W$ is closed and connected, the quotient group $G/W$ is also equivalent to a torus. Clearly $\dim W$ and $\dim G/W$ are both less than $n$.  Denoting by  $\mu_W$ and $\mu_{G/W}$ the associated Haar probability measures, and applying the quotient integral formula, we have
\[ \int_G \prod_{j=1}^m f_j\circ \phi_j d\mu_G
= \int_{G/W} \Big(\int_W \prod_{j=1}^m f_j\circ \phi_j (x+y) d\mu_W(y)\Big) d\mu_{G/W}(\xi).\]
For each $\xi\in G/W$ we fix a representative $x_\xi \in G$ of $\xi$ and  define $f_{j,\xi}(y_j) =f_j(y_j +\phi_j(x_\xi))$ for  $y_j\in \phi_j(W)$. Then if $y\in W$
\[ f_j\circ \phi_j(x_\xi +y) = f_j(\phi_j(y) +\phi_j(x_\xi)) =f_{j,\xi}\circ\phi_j|_W(y),\]
where $\phi_j|_W$ is the restriction of $\phi_j$ to $W$.
Note that $\phi_j|_W$ is a homomorphism of $W$ onto $U_j:= \phi_j(W)$. Further, for any closed subgroup of $W$ the inequality \eqref{dim'} holds with respect to $\phi_j|_W$, since $W$ is critical. Thus the induction assumption yields
\begin{equation}\label{tem_to}
\begin{aligned}
\int_G \prod_{j=1}^m f_j\circ \phi_j \,d\mu_G
&\le \int_{G/W} \prod_{j=1}^m \|f_{j,\xi}\|_{L^{p_j}(U_j)} \,d\mu_{G/W}(\xi)\\
&= \int_{G/W} \prod_{j=1}^m \Big( \int_{U_j} f_{j} (y_j+\phi_j(x_\xi) )^{p_j} d\mu_{U_j}(y_j)\Big)^{1/p_j} \,d\mu_{G/W}(\xi).
\end{aligned}\end{equation}
Here $\mu_{U_j}$ is the Haar probability measure on $U_j$.
We now define a group homomorphism $\varphi_j :G/W\to G_j/U_j$ by $\varphi_j(\xi) =\phi_j(x_\xi)+U_j$, $\xi \in G/W,$  and set $F_j(\xi_j) =\| f_j(\cdot \,+x_j)\|_{L^{p_j}(U_j)}$ for any $\xi_j\in G_j/U_j$  containing $x_j\in G_j$. Then the right hand-side of \eqref{tem_to} may be written as
\[ \int_{G/W} \prod_{j=1}^mF_j\circ\varphi_j (\xi)  d\mu_{G/W}(\xi).\]
We claim that
\begin{equation}\label{quo}
\codim_{G/W} V\ge\sum_{j=1}^m p_j^{-1} \codim_{G_j/U_j} \varphi_j(V)
 \end{equation}
holds for any closed subgroup $V$ of $G/W$.
Accepting this momentarily, since $\dim G/W <\dim G$, and applying the induction assumption again, we have
\begin{align*}
\int_{G/W} \prod_{j=1}^mF_j\circ\varphi_j (\xi)  d\mu_{G/W}(\xi) &\le \prod_{j=1}^m \Big(\int_{G_j/U_j} \Big( \int_{U_j} f_j(x_j+y_j)^{p_j} d\mu_{U_j}(y_j) \Big) d\mu_{G_j/U_j}(\xi_j)\Big)^{1/p_j}\\
&=\prod_{j=1}^m \| f_j\|_{L^{p_j}(G_j)}.
\end{align*}
Here the last equality follows from the quotient integral formula. Combining this  with \eqref{tem_to}, we obtain the desired estimate \eqref{BLtorus}.

It remains to prove \eqref{quo}. For a closed subgroup $V$ of $G/W$ we write $\mathfrak V =\mathfrak q^{-1}(V)$, where $\mathfrak q: G\to G/W$ is the quotient map. Since $V=\mathfrak V/W$ and $\varphi_j(V)=\phi_j(\mathfrak V)/U_j$, 
\[ \codim_{G/W} V= \codim _G \mathfrak V,\]
by the quotient manifold theorem.
Hence \eqref{dim'} implies that
\begin{align*}
\codim_{G/W} V
& \ge \sum_{j=1}^m p_j^{-1} (\codim_{G_j} \phi_j(\mathfrak V))\\
&= \sum_{j=1}^m p_j^{-1} \Big(\dim{G_j} -  \dim \varphi_j( V) -\dim U_j\Big)\\
&=\sum_{j=1}^m p_j^{-1} \codim_{G_j/U_j} \varphi_j( V) ,
 \end{align*}
by a further application of the quotient manifold theorem.

We next consider the case where each connected closed proper subgroup $W$ of $G$ gives rise to strict inequality in \eqref{dim'}.  In this case we define a closed convex set $\mathcal C$ by
\begin{align*} 
\mathcal C=\Big\{ t \in [0,1]^m : \codim_G &\,W\ge \sum_{j=1}^m t_j\, \codim_{G_j}\phi_j(W),\\ &\,\text{for all  connected closed subgroups } W\Big\}.
\end{align*}
By multilinear interpolation it is enough to show that for any extreme point  $t\in \mathcal C$,
\begin{equation}\label{tBL}
 \int_{G} \prod_{j=1}^m f_j\circ \phi_j\, d\mu_G \le \prod_{j=1}^m \|f_j\|_{L^{1/t_j}(G_j)}.
 \end{equation}
If $t$ is such a point then it satisfies at least one of the following: (i) there is a  proper connected critical subgroup $W$ of $G$ with respect to $t$; (ii) at least one of the $t_j$ is equal to $0$; 
(iii) $m=1$. From the above argument we see that \eqref{tBL} holds for the first case. In the third case we have $\|f_1\circ \phi_1\|_{L^1(G)}=\|f_1\|_{L^1(G_1)}$,
and so \eqref{tBL} holds.   In the second case we simply appeal to induction  on the number of components $m$. Hence \eqref{tBL} holds for all extreme points $t$ of $\mathcal C.$ \\

We now deal with the case where $G$ is not connected. In what follows we write $(y,z)\in \T^n\times F$,  $(y_j,z_j)\in \T^{n_j}\times F_j$ and $\phi_j(y,z)=(\phi_j^1(y,z),\phi_j^2(y,z))\in \T^{n_j}\times F_j$ for all $1\le j\le m.$  Since each $\phi_j^2$ is continuous it depends only on the variable $z$, and so for convenience we redefine it as a function on $F$.  For a nonnegative measurable function $f_j$ on $\T^{n_j}\times F_j$ and $z\in F$ we let $g_{j,z}(y_j) =f_j(y_j+\phi_j^1(0,z),\phi_j^2(z))$ for each $y_j\in \T^{n_j}$. Since
\[ f_j(\phi_j(y,z)) =f_j(\phi_j^1(y,0)+\phi_j^1(0,z),\phi_j^2(z))=g_{j,z}(\phi_j^1(y,0)),\]
we have
\[ \int_{G}\prod_{j=1}^m f_j\circ \phi_j (x) d\mu_G(x) =  \frac1{|F|}\sum_{z\in F} \int_{\T^n} \prod_{j=1}^m g_{j,z}(\phi_j^1(y,0)) dy. \]
Here $dy$ is the usual Lebesgue measure on $\T^n$.
The function $\psi_j(\cdot):=\phi_j^1(\cdot, 0)$ is a homomorphism of $\T^n$ into $\T^{n_j}$, and $\dim \phi_j(W\times \{0\}) =\dim \psi_j (W)$ for any closed subgroup $W$ of $\T^n$, and hence \eqref{dim'} is satisfied with respect to the data $(\T^n, \T^{n_j},\psi_j)$. Applying the case we have already established (for connected $G$), we have that
\begin{align*}
\frac1{|F|} \sum_{z\in F}\int_{\T^n} \prod_{j=1}^m g_{j,z}(\psi_j(y)) dy
&\le \frac1{|F|}\sum_{z\in F} \prod_{j=1}^m\Big( \int_{\T^{n_j}} |f_j(y_j+\phi^1_j(0,z),\phi^2_j(z))|^{p_j}  dy_j\Big)^{1/p_j}\\
&=\frac1{|F|}\sum_{z\in F} \prod_{j=1}^m\Big( \int_{\T^{n_j}} |f_j(y_j,\phi^2_j(z))|^{p_j}  dy_j\Big)^{1/p_j}.
\end{align*}

Note that $\phi_j^2 : F\to F_j$ is a  homomorphism of finite groups, and the induced measure on $F_j$ from $\mu_{G_j}$ is normalised counting measure for each $j$. Thus, applying Theorem  \ref{discrete} we have
\begin{align*}
\frac1{|F|}\sum_{z\in F} \prod_{j=1}^m&\Big( \int_{\T^{n_j}} |f_j(y_j,\phi^2_j(z))|^{p_j}  dy_j\Big)^{1/p_j}\\
&\le |F|^{-1} \BL_{\mathbf G}( H, \mathbf p) \prod_{j=1}^m \Big( \sum_{z_j\in F_j} \int_{\T^{n_j}} |f_j(y_j,z_j)|^{p_j}  dy_j\Big)^{1/p_j}\\
&=|F|^{-1} \BL_{\mathbf G}( H, \mathbf p) \Big(\prod_{j=1}^m|F_j|^{1/p_j}\Big) \prod_{j=1}^m\|f_j\|_{L^{p_j}(G_j)},
\end{align*}
where $H=\big\{(\phi_1^2(z),\cdots, \phi_m^2(z)) :z\in F\big\}$ is a subgroup of $ F_1\times \cdots\times F_m$. Combining all of the above estimates, we obtain
\[
\int_{G}\prod_{j=1}^m f_j\circ \phi_j (x) d\mu_G(x)\le \mathscr C(\phi, \mathbf p)\prod_{j=1}^m \|f_j \|_{L^{p_j}(G_j)}
\]
 with the constant $\mathscr C(\phi, \mathbf p)$ given by
\[ \mathscr C(\phi, \mathbf p) = \max\Big\{ |F|^{-1} |W| \prod_{j=1}^m \Big(|F_j|^{-1}|\phi^2_j(W)|\Big)^{-1/p_j} : W \text{ is a subgroup of }F\Big\}.\]

Now it remains to show that $\mathscr C(\phi, \mathbf p)$ is the smallest constant for \eqref{BLtorus} if \eqref{BLtorus} holds with a finite constant. By Theorem \ref{discrete}  there are subgroups $\Gamma_j$ of $F_j$, $1\le j\le m$, such that $\mathbf h=(h_j)$ with $h_j^{p_j}=|F_j||\Gamma_j|^{-1}\chi_{\Gamma_j}$ is an extremiser of the Brascamp--Lieb inequality associated with $(H,\mathbf p)$, that is to say
\[\sum_{z\in F} \Big(\prod_{j=1}^m h_j\circ \phi^2_j(z)\Big) =\BL_{\mathbf G}(H,\mathbf p)\prod_{j=1}^m \Big(\sum_{z_j\in F_j}h_j(z_j)^{p_j}\Big)^{1/p_j}.
\]
Let us set $f_j:= \chi_{\T^{n_j}}\otimes h_j$ on $\T^{n_j}\times F_j$. Then $\|f_j\|_{L^{p_j}(G_j)}=1$  and
\begin{align*}
\int_{G}\prod_{j=1}^m f_j\circ \phi_j (x) d\mu_G(x)
&= |F|^{-1}\sum_{z\in F} \Big(\prod_{j=1}^m h_j\circ \phi_j^2 (z)\Big)\\
&=|F|^{-1}\BL_{\mathbf G}(H,\mathbf p)\prod_{j=1}^m  |F_j|^{1/p_j}\\
&=\mathscr C(\phi, \mathbf p) \|f_j\|_{L^{p_i}(G_j)}.
\end{align*}
Hence $\mathscr C(\phi, \mathbf p)$ is the smallest constant for \eqref{BLtorus}.
\end{proof}

\subsubsection*{Remark} By the lifting lemma it is possible to find a sufficient condition for the finiteness of the Brascamp--Lieb constant appearing in the statement of Theorem \ref{re_torus} from that for the so-called \textit{localised Brascamp--Lieb constant} in \cite[Theorem 2.2]{BCCT2}, avoiding the above argument. Indeed, when $G=\T^n$, $G_j=\T^{n_j}$ and $\phi_j$ is surjective, there exists a linear transformation $\psi_j:\mathbb R^n\to \mathbb R^{n_j}$ such that $q_j\circ \psi_j=\phi_j\circ q$, where $q:\mathbb R^n\to T^n$ and $q_j:\mathbb R^{n_j}\to \T^{n_j}$ are the quotient maps. So, for any non-negative measurable function $f_j$  on $\T^{n_j}$,  the inequality \eqref{BLtorus} coincides with
\begin{equation*}
\int_{[-\frac12,\frac12)^n}\prod_{j=1}^m g_j\circ \psi_j (x) \,dx \le C \prod_{j=1}^m  k_j^{-1/p_j}\|g_j\|_{L^{p_j}(\mathbb R^{n_j})}.
\end{equation*}
Here $g_j =\widetilde f_j|_{B_j}$ where $\widetilde f_j$  is the periodic extension of $f_j$ and $B_j= \psi_j([-\frac12,\frac12)^n)$. Since  $B_j$ is a $k_j$-fold covering of $\T^{n_j}$ for some $k_j\ge1$,  we have that $\|g_j\|^{p_j}_{L^{p_j}(\R^{n_j})} = k_j \|f_j\|^{p_j}_{L^{p_j}(\T^{n_j})}$.  Hence, by Theorem 2.2 of  \cite{BCCT2},  \eqref{BLtorus} holds with finite  $C>0$ if  \eqref{dim'} holds for any subgroup of the form $W=q(V)\subset \T^n$ for some subspace $V$ of $\mathbb R^n.$
However, this reduction to the localised Brascamp--Lieb inequality does not appear to easily provide information on the precise Brascamp--Lieb constant, as is required by Theorem \ref{re_torus}.

\section{Further lines of enquiry and remarks}\label{further}
\subsection{Finiteness of the Brascamp--Lieb constant on general LCA groups}
As mentioned in the introduction -- see \eqref{conj} and \eqref{dualityidgen} -- it seems natural to expect that Theorems \ref{basicthm} and \ref{dualitythm} have a common generalisation to the setting of arbitrary LCA groups.  Given our approaches to Theorems \ref{basicthm} and \ref{dualitythm}, the natural first step would be to establish a finiteness criterion for the Brascamp--Lieb constant at this level of generality.

\subsection{The Brascamp--Lieb inequality on noncommutative groups}\label{finite}
While the definition of the Brascamp--Lieb constant $\BL_{\mathbf G} (H,\mathbf p)$ given in \eqref{BLabstract} readily generalises to general (unimodular\footnote{We include this hypothesis merely for convenience, allowing us to avoid making choices of Haar measures.}) locally compact groups, the complexities in defining the Fourier transform in this setting make the prospect of establishing an associated  Fourier duality principle of the form \eqref{conj} (for example) appear rather unclear. However, analogues of Theorems \ref{discrete} and \ref{torus} remain quite plausible. We illustrate this with the following theorem in the case of general \emph{finite} groups.
\begin{proposition}\label{optimal_f} Let $G_j$ be a finite group, $j=1,\cdots,m,$ and  $H$ be a subgroup of $G_1\times  \cdots\times G_m$. Assume that $G_j$ and $H$ are equipped with the counting measures $\mu_{G_j}$ and $\mu_H$, respectively.  For $\mathbf p=(p_j)\in [1,\infty]^m$, we have
\begin{equation}\label{cons_f}
\BL_{\mathbf G} (H,\mathbf p) 
=\max_{V_j\subset G_j} \Big|\int_H (a_1\chi_{V_1})\otimes \cdots\otimes (a_m\chi_{V_m}) d\mu_H \Big|,
\end{equation}
where the maximum is taken over all subgroups $V_j$ of $G_j$ and $a_j=|V_j|^{-1/p_j}$, $j=1,\cdots,m$.
\end{proposition}

\subsubsection*{Remark} 
The Brascamp--Lieb constant \eqref{cons_f} is given by
\begin{equation*}\label{pro_fin}
\begin{aligned}
\BL_{\mathbf G} (H,\mathbf p) 
&= \max\Big\{\Big|\bigcap_{j=1}^m \pi_j^{-1}(V_j)\Big| \prod_{j=1}^m (|V_j|)^{-1/p_j} : V_j \text{ is a subgroup of }G_j \Big\}\\
&= \max\Big\{  |V| \prod_{j=1}^m (|\pi_j(V)|)^{-1/p_j} : V \text{ is a subgroup of }G\Big\}.
\end{aligned}
\end{equation*}

Before proving Proposition \ref{optimal_f},  we introduce some convenient notation. 
For each $1\le j\le m$, let $G$ and $G_j$ be finite  groups,  
$\varphi_j:G\rightarrow G_j$ be homomorphisms, and
\[
\BL(\mathbf{\varphi},\mathbf{s};\mathbf{f})=\frac{\int_G\prod_{j=1}^m f_j(\varphi_j(x))^{s_j}d\mu(x)}{\prod_{j=1}^m(\int_{G_j}f_j(x_j) d\mu_j(x_j))^{s_j}}.
\]
Here $\mathbf{\varphi}=(\varphi_j)$, $\mathbf{s}=(s_j)\subset[0,1]^m$, $\mathbf{f}=(f_j)$, and all integrals are with respect to the counting measures $\mu$ and $\mu_j$ on $G$ and $G_j$, respectively. The mapping $\BL(\mathbf{\varphi},\mathbf{s};\cdot)$ is often referred to as the \textit{Brascamp--Lieb functional}, and
\begin{equation}\label{lieb}
\BL(\varphi,\mathbf{s}):=\sup_{\mathbf{f}\neq 0}\BL(\varphi,\mathbf{s};\mathbf{f})\end{equation} 
the \emph{Brascamp--Lieb constant} for the data $(\varphi,\mathbf{s})$. 
Proposition \ref{optimal_f} may now be restated as follows.
\begin{proposition}\label{thm_christ} Let $G$ and $G_j$ be finite  groups. Then
\begin{equation*}\label{lieb}
\BL(\varphi,\mathbf{s})
=\max_{\mathbf f} \BL(\varphi,\mathbf{s};\mathbf{f}), 
\end{equation*} 
where the maximum is taken over $f_j =a_j \chi_{\Gamma_j}$,  where  $\Gamma_j $  is a subgroup of $G_j$, and $a_j>0$.
\end{proposition}
In the remainder of this section we prove Proposition \ref{thm_christ}. The key ingredients are an abstract form of an inequality of K. Ball \cite{Ball} (see also \cite{BCCT1}) and properties of iterated convolutions of a function on a finite  group (the forthcoming Lemma \ref{limit}).
\begin{prop}[Ball's inequality]\label{ball} We have
$$
\BL(\varphi,\mathbf{s};\mathbf{f})\BL(\varphi,\mathbf{s};\mathbf{g})
\leq \BL(\varphi,\mathbf{s})
\BL(\varphi,\mathbf{s};\mathbf{f}*\mathbf{g}),$$
where $\mathbf{f}*\mathbf{g}=(f_j*g_j)$.
\end{prop}
 As in the euclidean case, Ball's inequality guarantees that a convolution of extremisers of the Brascamp--Lieb functional is also an extremiser -- see \cite{BCCT1}.
\begin{proof}
By homogeneity it is enough to prove Proposition \ref{thm_christ} when $\int_{G_j} f_j d\mu_j=\int_{G_j} g_j d\mu_j=1$ for each $j$.
Now, by the translation-invariance of the measures,
\begin{eqnarray*}
\begin{aligned}
\BL(\varphi,\mathbf{s};\mathbf{f})\BL(\varphi,\mathbf{s};\mathbf{g})&=\int_G\int_G\prod_{j=1}^m f_j(\varphi_j(y))^{s_j}g_j(\varphi_j(xy^{-1}))^{s_j}d\mu(x)d\mu(y)\\
&=\int_G\int_G\prod_{j=1}^m f_j(\varphi_j(y))^{s_j}g_j(\varphi_j(x)\varphi_j(y)^{-1})^{s_j}d\mu(y)d\mu(x)\\
&=\int_G\int_G\prod_{j=1}^m h_j^{x}(\varphi_j(y))^{s_j}d\mu(y)d\mu(x),
\end{aligned}
\end{eqnarray*}
where $h_j^{x}(z)=f_j(z)g_j(\varphi_j(x)z^{-1})$.
Hence
\begin{eqnarray*}
\begin{aligned}
\BL(\varphi,\mathbf{s};\mathbf{f})\BL(\varphi,\mathbf{s};\mathbf{g})&\leq  \BL(\varphi,\mathbf{s})\int_G\prod_{j=1}^m \left(\int_{G_j}h_j^{x}(z_j)d\mu_j(z_j)\right)^{s_j}d\mu(x)\\
&= \BL(\varphi,\mathbf{s})\int_G\prod_{j=1}^m(f_j*g_j(\varphi_j(x)))^{s_j}d\mu(x)\\
&=\BL(\varphi,\mathbf{s})
\BL(\varphi,\mathbf{s};\mathbf{f}*\mathbf{g}).
\end{aligned}
\end{eqnarray*}
\end{proof}
\begin{lemma}\label{limit} Let $G$ be a finite group with the counting measure. Let $f$ be a  non-negative function on $G$ with $\|f\|_{L^1(G)}=1$ and $f^{(\ell)}=f\ast  \cdots\ast f $ be the $\ell$-fold convolution of $f$ with itself. Then there is a subgroup $\Gamma$ of $G$ and  a positive integer $k_\circ$ such that  $f^{(\ell k_\circ)}$ converges to  $1/|\Gamma|\,\chi_\Gamma$ as $\ell\to \infty.$ Here $\chi_\Gamma$ is the characteristic function of $\Gamma.$
\end{lemma}
\begin{proof}[Proof of Proposition \ref{thm_christ}]
Clearly, $\BL(\varphi,\mathbf{s})\ge \sup \BL(\varphi,\mathbf{s};\mathbf{f}), 
$ for any  $\mathbf f=(a_j\chi_{\Gamma_j})$. It suffices to prove that  the reverse inequality also holds. Note that $f_j:G_j\to [0,\infty)$ is an element of $[0,\infty)^{n_j}$, where $n_j$ is the cardinality of $G_j$, hence the set $\{f_j:G_j\rightarrow [0,\infty):\| f_j \|_{p_j}=1\}$ is  compact. Hence the Brascamp--Lieb constant $$\BL(\varphi,\mathbf{s})=\sup_{\|f_j\|_{L^{1}(G_j)}=1}\BL(\varphi,\mathbf{s};\mathbf{f})$$
 is always finite for any $1\le p_j\le \infty$, and an extremiser $\mathbf{f}$ exists.
Let $\mathbf f$ be an extremiser with  $\|f_j\|_{L^1(G_j)}=1$. Let us set $\mathbf f^{(\ell)}=(f_j^{(\ell)})$, with $f_j^{(\ell)}$ the $\ell$-fold convolution of $f_j$. Since each $f_j$ is non-negative, we apply Lemma \ref{limit} to obtain $\Gamma_j \subset G_j$ and $k_j\in \mathbb N$ such that $f_j^{(\ell k_j) }$ converges to $|\Gamma_j|^{-1} \chi_{\Gamma_j}$ as $\ell\to \infty$. Note that 
$\mathbf f^{(\ell k_1\cdots k_m)}$, $\ell\in\N,$  are extremisers of the Brascamp--Lieb functional by Ball's inequality. Letting $\ell \to \infty$, $\BL(\varphi,\mathbf s)=\BL(\varphi,\mathbf s;\mathbf g)$, where $g_j=|\Gamma_j|^{-1}\chi_{\Gamma_j}$ for each $j$.  
\end{proof}

It remains to verify Lemma \ref{limit}, which follows by a routine application of the ergodic theorem for a random walk on a finite group -- formulated in an appropriate abstract setting in \cite[Proposition 9.9.6]{SST}. Here we appeal to the following simplified statement.
\begin{theorem}[\cite{SST}]\label{ergodic}
Suppose $G$ is a finite group. Let $g$ be a nonnegative function on $G$ with $\|g\|_{L^1(G)}=1$ and set $\Delta =\{x\in G : g(x)>0\}$, the support of $g$.  Then $g^{(k)}\to a\chi_G$ as $k\to \infty$ if and only if there exists $k\in \N$ such that 
$G=\Delta\cdots \Delta$ with $k$-factors.
Here $\|a\chi_G\|_{L^1(G)}=1$ and $g^{(k)}$ is the $k$-fold convolution of $g$.
\end{theorem}

\begin{proof}[Proof of Lemma \ref{limit}]
We fix $x_\circ$ in the support of $f$ and let  $\mathfrak C(f,G)$ be the collection of subgroups of $G$  containing $\Delta x^{-1}_\circ:=\{xx_\circ^{-1}\in G: f(x)>0\}.$  Clearly $\mathfrak C(f,G)$ is nonempty since $G\in \mathfrak C(f,G)$. Suppose that $\Gamma$ is the smallest subgroup of $G$ contained  in $\mathfrak C(f,G)$.  
We claim that it suffices to show that 
 \begin{equation}\label{conv}
 \lim_{\ell\to \infty} f^{(\ell)} =1/|\Gamma|\,\chi_\Gamma
 \end{equation}
when $x_\circ$ is the identity element of $G$. To see this suppose that  $x_\circ$ is not the identity element, and observe that we can find a positive integer $k_\circ$ such that $ x_\circ^{k_\circ} $ is the identity, since $G$ is  finite. Note that $f^{(k_\circ)}(x_\circ^{k_\circ})\neq 0$. This gives \eqref{conv} with $f^{(k_\circ)}$ instead of $f$, as required.

We now show \eqref{conv} when $x_\circ$ is the identity element. Since the support of $f$ is contained in $\Gamma,$ $f=f|_\Gamma$ is a normalised function on the finite group $\Gamma$. So, if there is a positive integer $\ell_\circ$ such that the $\ell_\circ$-fold product
\[ \Delta \cdots \Delta =\supp f^{(\ell_\circ)}=\Gamma,\]
then \eqref{conv} follows from Theorem \ref{ergodic}. Thus matters are reduced to finding such an $\ell_\circ.$
By the smallness assumption on $\Gamma$, it is straightforward to see that 
\begin{equation}\label{strict}
|\supp f^{(\ell)}|<|\supp f^{(2\ell)}|\le |\Gamma|,
\end{equation}
whenever $\Delta_\ell:=\supp f^{(\ell)}$ is a proper subset of $\Gamma$. Indeed, since the identity element belongs to $\Delta_\ell$, we have that $\Delta_\ell \subset \Delta_{\ell+1}$ and  $|\Delta_\ell|$ is increasing. So, if  \eqref{strict} did not hold, we would have $\Delta_\ell=\Delta_\ell\Delta_\ell$, and so $\Delta_\ell$ would be a proper subgroup of $ \Gamma$ containing the support of $f$, contradicting the smallness of $\Gamma$.
Thus from \eqref{strict} and the finiteness of $\Gamma$, we find a positive integer $\ell_\circ$ such that $\Delta_{\ell_\circ}=\Gamma$.
\end{proof}

\subsection{Fourier duality for the Brascamp--Lieb inequality on Lorentz spaces}
There are further generalisations of the classical Brascamp--Lieb inequality that permit duality statements of the type we present in this paper. The so-called \textit{Lorentz--Brascamp--Lieb inequality}, which also encompasses fractional integral inequalities (such as the Hardy--Littlewood--Sobolev inequality), is a clear example.
Let us denote by $\BL(H,\mathbf p, \mathbf r)$ the smallest constant $C$ for which
\begin{equation}\label{BLL}
\left|\int_H f_1\otimes\cdots\otimes f_m\right|\leq C\prod_{j=1}^m\|f_j\|_{L^{p_j,r_j}(H_j)}
\end{equation}
holds for all $f_j\in L^{p_j,r_j}(H_j)$. Here, as in the classical situation,  the $H_j$ are euclidean spaces and $H$ is a subspace of $ H_1\times \cdots\times H_m$. The exponents $\mathbf p=(p_j)\in (1,\infty)^m$, $\mathbf r=(r_j)\in  [1,\infty]^m$ and $L^{p_j,r_j}(H_j)$ is the standard Lorentz space.  Finiteness of $\BL(H,\mathbf p, \mathbf r)$ for general (Lorentz--Brascamp--Lieb) data was  studied by Christ \cite{Perry}, proving that  if $(H,\mathbf p)$ is simple (in that it admits no proper critical subspace) then \eqref{BLL} holds whenever
\begin{equation}\label{tem}
\sum_{j=1}^m \frac1{r_j} \ge 1.
\end{equation}
For the opposite direction, Bez, Lee, Nakamura and Sawano \cite{BLNS} showed that if the Lorentz refinement \eqref{BLL} holds then necessarily the scaling and dimension conditions (\eqref{scale}, \eqref{dim}) and \eqref{tem} hold.\footnote{It was also observed in \cite{BLNS} that \eqref{scale},\eqref{dim} and \eqref{tem} are not sufficient for \eqref{BLL}, leaving the question of finiteness in \eqref{BLL} open in general.}

Combining these results with Theorem \ref{hilbert} we  obtain a simple duality property for Lorentz--Brascamp--Lieb data.
\begin{corollary}\label{lorentz} Let $H\subseteq H_1\times\cdots\times H_m$ and let $p_j\in(1,\infty)$ and $r_j\in[1,\infty]$ for each $1\le j\le m.$ If $r_j\le \min\{p_j, p_j'\}$ for all $j$, then $\BL(H,\mathbf p,\mathbf r)$ is finite if and only if $\BL(H^\perp,\mathbf p',\mathbf r)$ is finite. In addition, when $(H,\mathbf p)$ is simple,  $\BL(H,\mathbf p,\mathbf r)$ is finite if and only if $\BL(H^\perp,\mathbf p',\mathbf r)$ is finite, without  restriction on the $r_j$.
\end{corollary}

In particular, Corollary \ref{lorentz} applied with $r_j=1$ establishes a duality between certain measure-theoretic statements.  Namely, for $p_j\in(1,\infty)$,
\[ | H\cap (A_1\times \cdots\times A_m)|\lesssim \prod_{j=1}^m |A_j|^{\frac1{p_j}}\quad\text{for any } A_j\subset H_j\]
 if and only if
\[ | H^{\perp}\cap (A_1\times \cdots \times A_m)|\lesssim \prod_{j=1}^m |A_j|^{\frac1{p_j'}}\quad\text{for any }  A_j\subset H_j.\]


\end{document}